\documentclass[11pt]{amsart}
\usepackage{amssymb, paralist, xspace, graphicx, url, amscd, euscript, mathrsfs,stmaryrd,epic,eepic,color}
\usepackage[all]{xy}
\SelectTips{cm}{}

\numberwithin{equation}{section}

1

\numberwithin{subsection}{section}

\allowdisplaybreaks[1]

\newtheorem*{namedtheorem}{\theoremname}
\newcommand{\theoremname}{testing}

\newtheorem{theorem}[subsection]{Theorem}
\newtheorem{proposition}[subsection]{Proposition}
\newtheorem{proposition-definition}[subsection]
{Proposition-Definition}
\newtheorem{corollary}[subsection]{Corollary}
\newtheorem{lemma}[subsection]{Lemma}

\theoremstyle{definition}
\newtheorem{definition}[subsection]{Definition}

\newtheorem{examples}[subsection]{Examples}
\newtheorem{remark}[subsection]{Remark}

\theoremstyle{remark}

\newcommand\cC{\mathcal{C}}

\newcommand\cG{\mathcal{G}}

\newcommand\cM{\mathcal{M}}
\newcommand\cN{\mathcal{N}}
\newcommand\cO{\mathcal{O}}

\newcommand\cX{\mathcal{X}}

\newcommand\ocM{\overline{\mathcal{M}}}
\newcommand\ocX{\overline{\mathcal{X}}}

\newcommand\DD{\mathbb{D}}

\newcommand\FF{\mathbb{F}}
\newcommand\GG{\mathbb{G}}

\newcommand\NN{\mathbb{N}}

\newcommand\PP{\mathbb{P}}
\newcommand\QQ{\mathbb{Q}}

\newcommand\VV{\mathbb{V}}

\newcommand\ZZ{\mathbb{Z}}

\newcommand\fG{\mathfrak{G}}

\newcommand\fM{\mathfrak{M}}

\newcommand\fS{\mathfrak{S}}

\renewcommand\frm{\mathfrak{m}}

\newcommand\sG{\mathscr{G}}
\newcommand\sN{\mathscr{N}}

\newcommand\arr{\ifinner\to\else\longrightarrow\fi}

\def\displaytimes_#1{\mathrel{\mathop{\times}\limits_{#1}}}

\def\displayotimes_#1{\mathrel{\mathop{\bigotimes}\limits_{#1}}}

\newcommand\Aut{\operatorname{Aut}}

\newcommand\pic{\operatorname{Pic}}

\newcommand\Spec{\operatorname{Spec}}

\newdir{ >}{{}*!/-5pt/@{>}}

\newcommand\doublelong[2]{\mathbin{\xymatrix{{}\ar@<3pt>[r]^{#1}
\ar@<-3pt>[r]_{#2}&}}}

\newlength{\ignora}

\newcommand{\sm}{{\mathrm{sm}}}
\newcommand{\sh}{{\mathrm{sh}}}

\newcommand{\Dan}[1]{}
\newcommand{\Matthieu}[1]{}

\renewcommand{\setminus}{\smallsetminus}

\theoremstyle{plain}

\theoremstyle{definition}

\newcommand{\Hom}{\underline {\text{\rm Hom}}}

\numberwithin{equation}{subsection}

\def\bin#1#2{\mbox{\scriptsize $\big\{\!\!\!\begin{array}{c} #1 \\ #2
    \end{array}\!\!\!\big\}$\normalsize}}

\begin{document}

\title{Moduli of Galois $p$-covers in mixed characteristics}

\author[Abramovich]{Dan Abramovich}

\author[Romagny]{Matthieu Romagny}

\address[Abramovich]{Department of Mathematics\\
Brown University\\
Box 1917\\
Providence, RI 02912\\
U.S.A.}
\email{abrmovic@math.brown.edu}

\address[Romagny]{Institut de Math\'ematiques\\
Universit{\'e} Pierre et Marie Curie (Paris 6)\\
Case 82\\
4 place Jussieu\\
F-75252 Paris Cedex 05.}
\email{romagny@math.jussieu.fr}

\thanks{Abramovich supported in part by NSF grants  DMS-0603284 and DMS-0901278.}
\date{\today}

\maketitle

\tableofcontents

\section{Introduction}

Fix a prime number $p$.
The aim of this paper is to define a complete moduli stack of
degree-$p$ covers $Y \to \cX$, with $Y$ a stable curve which is a
 $\cG$-torsor over $\cX$, for a suitable group scheme
$\cG/\cX$. The curve $\cX$ is a twisted curve in the sense of
\cite{A-V,AOV2} but in general not stable. This follows the same general approach as the characteristic-0 paper \cite{ACV}, but diverges from that of \cite{AOV2}, where the curve $\cX$ is stable,
the group scheme $\cG$ is assumed linearly reductive, 
but $Y$ is in general much more singular. 
\Dan{Say something about singularities of moduli, about lifting problems, etc.}

The approach is based on \cite[Proposition 1.2.1]{Raynaud} of Raynaud,
and the more general notion  of {\em effective model} of a
group-scheme action due to the second author
\cite{Romagny}. The general strategy was outlined in \cite{A-Lubin} in
a somewhat special case.

\subsection{Rigidified group schemes}

The group scheme $\cG$ comes with a supplementary structure which
we call a {\em generator}. Before we define this notion, let us
briefly recall from Katz-Mazur \cite[\S 1.8]{KM} the concept of a
{\em full set of sections}. Let
$Z\to S$ be a finite locally free morphism of schemes of degree $N$.
Then for all affine $S$-schemes $\Spec(R)$, the $R$-algebra
$\Gamma(Z_R,\cO_{Z_R})$ is locally free of rank $N$ and has a canonical
norm mapping. We say that a set of $N$ sections $x_1,\dots,x_N\in Z(S)$
is a {\em full set of sections} if and only if for any affine $S$-scheme
$\Spec(R)$ and any $f\in \Gamma(Z_R,\cO_{Z_R})$, the norm of $f$
is equal to the product $f(x_1)\dots f(x_N)$.

\begin{definition} \label{dfn_generator}
Let $G\to S$ be a finite locally free group scheme of order $p$.
A {\em generator} is a morphism of $S$-group schemes
$\gamma:(\ZZ/p\ZZ)_S\to G$ such that the sections $x_i=\gamma(i)$,
$0\le i\le p-1$, are a full set of sections.
A {\em rigidified group scheme} is a group scheme of degree $p$
with a generator.
\end{definition}

The notion of generator is easily described in terms of the Tate-Oort
classification of group schemes of order $p$. This is explained
and complemented in appendix~\ref{complements_order_p}.

\begin{remark} \label{closure}
One can define the stack of rigidified group schemes a bit more directly:
consider the Artin stack $\fG\fS_p$ of group-schemes of degree $p$, and
let $\cG^u \to\fG\fS_p$ be the universal group-scheme - an object of $\cG^u$
over a scheme $S$ consists of a group-scheme $\cG \to S$ with a section
$S \to \cG$. It has a unique non-zero point over $\QQ$ corresponding to
$\ZZ/p\ZZ$ with the section $1$. The stack of rigidified group schemes
is canonically isomorphic to the closure of this point.

Of course describing a stack as a closure of a sub-stack is not ideal
from the moduli point of view, and we find the definition using
Katz--Mazur generators  more satisfying.
\end{remark}

\subsection{Stable $p$-torsors}

Fix a prime number $p$ and integers $g,h,n\ge 0$ with $2g-2+n>0$.

\begin{definition} \label{Def:stable-torsor}
A {\em stable $n$-marked $p$-torsor} of genus $g$ (over some base scheme $S$)
is a triple
$$(\cX,\cG,Y)$$
where
\begin{enumerate}
\item $(\cX,\{\Sigma_i\}_{i=1}^n)$ is an $n$-marked twisted curve
of genus $h$,
\item $(Y,\{P_i\}_{i=1}^n)$ is a nodal curve of genus $g$ with \'etale
marking divisors $P_i\to S$, which is stable in the sense of
Deligne-Mumford-Knudsen,
\item $\cG \to \cX$ is a rigidified group-scheme of degree $p$,
\item $Y \to \cX$ is a $\cG$-torsor and $P_i=\Sigma_i\times_\cX Y$ for all $i$.
\end{enumerate}
\end{definition}

Note that as usual the markings $\Sigma_i$ (resp. $P_i$) are required to
lie in the smooth locus of $\cX$ (resp. $Y$). They split into two groups.
In the first group $\Sigma_i$ is twisted and $[P_i:S]=1$, while in the second
group $\Sigma_i$ is a section and $[P_i:S]=p$. The number $m$ of
twisted markings is determined by $(2g-2)=p(2h-2)+m(p-1)$ and it is
equivalent to fix $h$ or $m$.

The notion of stable marked $p$-torsor makes sense over an
arbitrary base scheme $S$. Given stable $n$-marked $p$-torsors $(\cX, \cG, Y)$
over $S$ and $(\cX', \cG', Y')$ over $S'$, one defines as usual a morphism
$(\cX, \cG, Y) \to (\cX', \cG', Y')$ over $S\to S'$ as a fiber
diagram.  This defines a category fibered over $\Spec \ZZ$ that we
denote $ST_{p,g,h,n}$.

Our main result is:

\begin{theorem} \label{maintheorem}
The category $ST_{p,g,h,n}/\Spec \ZZ$ is a proper Deligne-Mumford
stack with finite diagonal.
\end{theorem}

Notice that $ST_{p,g,h,n}$ contains an open substack of {\'e}tale
$\ZZ/p\ZZ$-covers. Identifying the closure of this open locus remains
an interesting question.

\subsection{Organization} Section \ref{Sec:algebricity} is devoted to Proposition \ref{algebraicity}, in particular showing the algebricity of $ST_{p,g,h,n}$. Section \ref{Sec:properness} completes the proof of Theorem \ref{maintheorem} by showing properness. We give simple examples in Section \ref{Sec:examples}. Two appendices are provided - in Appendix \ref{complements_order_p} we discuss embeddings of group schemes of order $p$ into smooth group schemes. In Appendix \ref{Sec:Weil} we recall some facts about the Weil restriction of closed subschemes, and state the representability result in a form useful for us.

\subsection{Acknowledgements} We thank Sylvain Maugeais for helping us clarify a point in this paper.

\section{The stack $ST_{p,g,h,n}$} \label{Sec:algebricity}

In this section, we review some basic facts on twisted curves and then we
show:

\begin{proposition} \label{algebraicity}
The category $ST_{p,g,h,n}/\Spec \ZZ$ is an algebraic
stack of finite type over $\ZZ$.
\end{proposition}

\subsection{Twisted curves and Log twisted curves}
We review some material from Olsson's treatment in \cite[Appendix A]{AOV2},
with some attention to properness of the procedure of ``log twisting''.

Recall that a {\em twisted curve} over a scheme $S$ is a tame Artin
stack $\cC \to S$ with a collection of gerbes $\Sigma_i \subset \cC$
satisfying the following conditions: 

\begin{enumerate}
\item The coarse moduli space $C$ of $\cC$ is a prestable curve over
  $S$, and the images $\bar\Sigma_i$ of $\Sigma_i$ in $C$ are the
  images of disjoint sections 
  $\sigma_i: S   \to C$ of $C \to S$ landing in the smooth locus.
\item \'Etale locally on $S$ there are positive integers $r_i$ such that, on a
  neighborhood  of $\Sigma_i$ we can identify $\cC$ with the root stack
  $C(\sqrt[r_i]{\bar\Sigma_i})$. 
\item Near a node $z$ of $C$  write $C^{sh} = \Spec
  (\cO_S^{sh}[x,y]/(xy-t))^{sh}$. Then there exists a positive integer
  $a_z$ and 
  an element $s \in \cO_S^{sh}$ such that $s^{a_z} = t$ and  
$$\cC^\sh = [\Spec \cO_S^{sh}[u,v]/(uv-s))^{sh} / \mu_{a_z}],$$
where $\mu_{a_z}$ acts via $(u,v) \mapsto (\zeta u , \zeta^{-1} v) $
and where $x=u^{a_z}$ and $y=v^{a_z}$. 
\end{enumerate}

The purpose of \cite[Appendix A]{AOV2} was to show that twisted curves
form an Artin stack which is locally of finite type over $\ZZ$. There
are two steps involved.

The
introduction of the stack structure over the markings is a
straightforward step: the stack $\fM_{g,\delta}^{tw}$ of twisted curves
with genus $G$ and $n$ markings is
the infinite disjoint union $\fM_{g,\delta}^{tw}=\sqcup \fM_{g,\delta}^{r}$,
where $r$ runs over the possible {\em marking indices}, namely vectors
of positive integers $r=(r_1,\ldots,r_n)$, and the 
stacks $\fM_{g,\delta}^{r}$ are all isomorphic to each other - the
universal family over $\fM_{g,\delta}^{r}$ is obtained form that over
$\fM_{g,\delta}^{(1,\ldots,1)}$ by taking the $r_i$-th root of $\bar \Sigma_i$.

The more subtle point is the introduction of twisting at nodes. Olsson
achieves this using the canonical log structure of prestable curves,
and provides an equivalence between twisted curves with
$r=(1,\ldots,1)$ and log-twisted 
curves. A {\em log twisted curve} over a scheme $S$ is the data of a
prestable curve $C/S$ along with a simple extension $\cM^S_{C/S}
\hookrightarrow \cN$. Here $\cM^S_{C/S}$ is F. Kato's canonical locally free
log structure of the base $S$ of the family of  prestable curves
$C/S$, and a {\em simple extension} is an injective morphism $\cM^S_{C/S}
\hookrightarrow \cN$ of locally free log
structures of equal rank where an irreducible element is sent to a
multiple of an irreducible element up to units.

We now describe an aspect of this equivalence which is relevant for
our main results. Consider a family of prestable curves $C/S$ and
denote by $\iota: \operatorname{Sing} C/S \to C$ the embedding of the 
locus where $\pi: C\to S$ fails to be smooth. A {\em node function} is a
section $a$ of $\pi_*\iota_* \NN_{\operatorname{Sing} C/S}$. In other words it gives a
positive integer $a_z$ for each singular point $z$ of $C/S$ in a
continuous manner. Given a morphism $T \to S$, we say that  a twisted
curve $\cC /T$ with coarse moduli space 
$C_T$ is {\em $a$-twisted} over $C/S$  if the index of a node of $\cC$ over a
node $z$ of $C$ is precisely $a_z$. 

\begin{proposition}\label{Prop:twisting-proper}
Fix a family of prestable curves $C/S$ of genus $g$ with $n$ markings
over a noetherian scheme 
$S$. Further fix marking indices $r=(r_1,\ldots, r_n)$ and a node
function $a$. Then the category of $a$-twisted curves over $C/S$ with marking
indices given by $r$ is a proper and quasi-finite tame stack over $S$.  
\end{proposition}

\begin{proof}
The problem is local on $S$, and further it is stable under base
change in $S$. So it is enough to prove this when $S$ is a versal
deformation space of a prestable curve $C_s$  of genus $g$ with $n$ 
markings, over a closed geometric
point $s\in S$, in such a way that we have a  chart $\NN^k\to \cM^S_{C/S}$
of the log structure, where $k$ is the number of nodes of $C_s$.  
The image  of the $i$-th generator  of $\NN^k$ in $\cO_S$ is the
defining equation of the smooth divisor $D_i$ where the $i$-th node
persists. Now consider an $a$-twisted curve over $\phi:T\to S$,
corresponding to a simple extension
$\phi^*\cM^S_{C/S} \to \cN$ where the image of the $i$-th generator
$m_i$ becomes an $a_i$-multiple up to units. This precisely means that
$\cO_{C_T}^* m_i$, the principal bundle associated to $\cO_S(-D_i)$,
is an $a_i$-th power. In other words, the stack of $a$-twisted curves
over $C/S$ is isomorphic to the stack $$S(\sqrt[a_1]{D_1}\cdots
\sqrt[a_n]{D_n}) \ \ = \ \ S(\sqrt[a_1]{D_1})\, \times_S\cdots
\times_S\,S(\sqrt[a_n]{D_n}) $$ encoding $a_i$-th roots of
$O_S(D_i)$. This is evidently proper and quasi-finite tame stack over $S$. 
\end{proof}

We now turn to the index of twisted points in a stable $p$-torsor.

\begin{lemma} \label{index_bounded}
Let $(\cX, \cG, Y)$ be a stable $p$-torsor. Then the index
of a point $x\in\cX$ divides $p$.
\end{lemma}

\begin{proof}
Let $r$ be the index of $x$ and $d$ the local degree of $Y\to\cX$
at a point $y$ above $x$.
Since $Y\to \cX$ is finite flat of degree $p$ and $\cG$ acts
transitively on the fibers, then $d\mid p$. Let $f:\cX\to X$ be the
coarse moduli space of $\cX$. In order to compute $d$, we pass to
strict henselizations on $S$, $X$ and $Y$ at the relevant points.
Thus $S$ is the spectrum of a strictly henselian local ring $(R,\frm)$,
and we have two cases to consider.

If $x$ is a smooth point,
\begin{itemize}
\item $X\simeq \Spec R[a]^\sh$,
\item $Y\simeq \Spec R[s]^\sh$,
\item $\cX\simeq [D/\mu_r]$ with $D=\Spec R[u]^\sh$ and
$\zeta\in\mu_r$ acting by $u\mapsto \zeta u$.
\end{itemize}
Consider the fibered product $E=Y\times_\cX D$. The map $E\to Y$
is a $\mu_r$-torsor of the form $E\simeq \Spec \cO_Y[w]/(w^r-f)$
for some invertible function $f\in\cO_Y^\times$, and $E\to D$ is a
$\mu_r$-equivariant map given by $u\mapsto \varphi w$ for some function
$\varphi$ on $Y$.
Let $\tilde x:\Spec k\to D$ be a point mapping to $x$ in $\cX$ i.e.
corresponding to $u=\frm=0$, and let $\bar\varphi,\bar f$ be the
restrictions of $\varphi,f$ to $Y_{\tilde x}$. The preimage
of $\tilde x$ under $E\to D$ is a finite $k$-scheme with algebra
$k[s][w]/(\bar\varphi,w^r-\bar f)$.
We see that $d=r\dim_k k[s]/(\bar\varphi)$
and hence the index $r$ divides $p$.

If $x$ is a singular point, there exist $\lambda,\mu,\nu$
in $\frm$ such that
\begin{itemize}
\item $X\simeq \Spec (R[a,b]/(ab-\lambda))^\sh$,
\item $Y\simeq \Spec (R[s,t]/(st-\mu))^\sh$,
\item $\cX\simeq [D/\mu_r]$ where $D=\Spec (R[u,v]/(uv-\nu))^\sh$,
\end{itemize}
and $\zeta\in\mu_r$ acts by $u\mapsto \zeta u$ and $v\mapsto \zeta^{-1}v$.
The scheme $E=Y\times_\cX D$ is of the form
$E\simeq \Spec \cO_Y[w]/(w^r-f)$ for some invertible function
$f\in\cO_Y^\times$, and the map $E\to D$ is given by
$u\mapsto \varphi w$, $v\mapsto \psi w^{-1}$ for some functions
$\varphi,\psi$ on $Y$ satisfying $\varphi\psi=\nu$. Let
$\tilde x:\Spec k\to D$ be a point mapping to $x$
and let $\bar\varphi,\bar\psi,\bar f$
be the restrictions of $\varphi,\psi,f$ to $Y_{\tilde x}$. The preimage
of $\tilde x$ under $E\to D$ is a finite $k$-scheme with algebra
$k[s,t][w]/(st,\bar\varphi,\bar\psi,w^r-\bar f)$.
We see that $d=r\dim_k k[s,t]/(st,\bar\varphi,\bar\psi)$
and hence $r$ divides $p$.
\end{proof}

\subsection{Proof of proposition~\ref{algebraicity}}

Let $\delta=(\delta_1,\dots,\delta_n)$ be the sequence of degrees of the
markings $P_i$ on the total space of stable $p$-torsors, with each
$\delta_i$ equal to $1$ or $p$.
We build $ST_{p,g,h,n}$ from existing stacks: the stack
$\ocM_{g,\delta}$ of Deligne-Mumford-Knudsen stable marked curves
(for the family of curves $Y$), the stack $\fM$ of twisted curves
(for the family of marked twisted curves $\cX$), and Hilbert schemes
and $\Hom$ stacks for construction of $Y \to \cX$ and $\cG$.

{\sc Bounding the twisted curves.} We have an evident forgetful
functor $ST_{p,g,h,n} \to \ocM_{g,\delta} \times \fM$. Note that the image
$ST_{p,g,h,n} \to \fM$ lies in an open substack $\fM'$ of finite type over
$\ZZ$: the index of the twisted curve $\cX$ divides $p$ by
Lemma~\ref{index_bounded}, and its
topological type is bounded
by that of $Y$. The stack $\fM'$ parametrizing such twisted curves
is of finite type over $\ZZ$ by \cite[Corollary A.8]{AOV2}.

Set $M_{Y,\cX} = \ocM_{g,\delta} \times \fM'$. This is an algebraic stack of
finite type over $\ZZ$.

{\sc The map $Y \to \cX$.} Consider the universal family $Y \to
M_{Y,\cX}$  of stable curves of genus $g$ and the universal family
$\cX \to M_{Y,\cX}$ of twisted curves, with associated family of coarse
curves $X \to M_{Y,\cX}$. Since Hilbert schemes of fixed Hilbert
polynomial are of finite type, there is an algebraic stack
$Hom^{\leq p}_{M_{Y,\cX}}({Y,X})$, of finite type over $M_{Y,\cX}$,
parametrizing morphisms $Y_s \to X_s$ of degree $\leq p$ between the
respective fibers. By \cite[Corollary C.4]{AOV2} the stack
$Hom^{\leq p}_{M_{Y,\cX}}({Y,\cX})$ corresponding to maps $Y_s\to \cX_s$
with target the twisted curve  is of finite type over
$Hom^{\leq p}_{M_{Y,\cX}}({Y,X})$, hence over $M_{Y,\cX}$. There is an open
substack $M_{Y\to\cX}$ parametrizing
flat morphisms of degree precisely $p$. We have an evident forgetful
functor $ST_{p,g,h,n}\to M_{Y\to\cX}$ lifting the functor $ST_{p,g,h,n} \to
\ocM_{g,\delta} \times \fM'$ above.

{\sc The rigidified group scheme $\cG$.} The scheme $Y_2 = Y\times_\cX Y$
is flat of degree $p$ over $Y$. Giving it the structure of a group scheme
over $Y$ with unit section equal to the diagonal $Y\to Y_2$ is tantamount
to choosing structure $Y$-arrows  $m:Y_2 \times_Y Y_2 \to Y_2$
and $i': Y_2 \to Y_2$, which are parametrized by a $Hom$-scheme, and
passing to the closed subscheme where these give a group-scheme
structure
(that this condition is closed follows from representability
of the Weil restriction, see the discussion in the appendix and in
particular Corollary \ref{equalizer}).
Giving a group scheme $\cG$
over $\cX$ with isomorphism $\cG \times_\cX Y \simeq Y_2$ is tantamount to
giving descent data for $Y_2$ with its chosen group-scheme
structure. This is again parametrized by a suitable
$Hom$-scheme. Finally requiring that the projection $Y_2 \to Y$
correspond to an action of $\cG$ on $Y$ is a closed condition
(again by Weil restriction, see Corollary \ref{equalizer}).

Passing to a suitable $Hom$-stack we can add 
a homomorphism  $\ZZ/p\ZZ \to \cG$, giving 
a section $\cX \to \cG$ (equivalently a morphism $\cX \to \cG^u$, see
remark~\ref{closure}).
By \cite[corollary~1.3.5]{KM}, the locus of the base where this
section is a generator is closed.
Since $Y_2\to Y$ and $Y \to \cX$ are finite, all the necessary
$Hom$ stacks are in fact of finite type. 

The resulting stack is clearly isomorphic to $ST_{p,g,h,n}$. 

\section{Properness} \label{Sec:properness}

Since $ST_{p,g,h,n}\to \Spec \ZZ$ is of finite type, we need to prove the
valuative criterion for properness. 

We have the following situation: 
\begin{enumerate}
\item $R$ is a discrete valuation ring with spectum $S = \Spec R$, fraction
  field $K$ with corresponding generic point $\eta = \Spec K$, and residue field
  $\kappa$ with corresponding special point $s=\Spec \kappa$.
\item $(\cX_\eta,\cG_\eta, Y_\eta)$ a stable marked
$p$-torsor of genus $g$ over $\eta$.
\end{enumerate}

By an {\em extension} of $(\cX_\eta, \cG_\eta, Y_\eta)$ across $s$ we mean 
\begin{enumerate}
\item a local extension $R \to R'$ with $K'/K$ finite,
\item a stable marked $p$-torsor $(\cX', \cG', Y')$ of genus $g$ over $S' =
  \Spec R'$, and
\item an isomorphism $(\cX', \cG', Y')_\eta' \simeq (\cX_\eta,
  \cG_\eta, Y_\eta) \times_\eta \eta'$.
\end{enumerate}

We have

\begin{proposition}
An extension exists. When extension over $S'$ exists, it is unique up
to a unique isomorphism.
\end{proposition}

\begin{proof}

{\sc Extension of $Y_\eta$.} Since $\ocM_{g,\delta}$ is proper, there is a
stable marked curve $Y'$ extending $Y_\eta$ over some $S'$, and this
extension is unique up to a unique isomorphism. We replace $S$ by $S'$,
and assume  that there is $Y$ over $S$ with generic fiber $Y_\eta$.

{\sc Coarse extension of $\cX_\eta$.} By unicity, the action of $G$ on
$Y_\eta$ induced by the map $G_{\cX_\eta}\to\cG_\eta$ extends to $Y$.
There is a finite extension $K'/K$ such that the
intersection points of the orbits of geometric irreducible components of
$Y_\eta$ under the action of $G$ are all $K'$-rational. We may and do
replace $S$ by the spectrum of the integral closure of $R$ in $K'$.
Let us call $Y_1,\dots,Y_m$ the orbits of irreducible
components of $Y$ and $\{y_{i,j}\}_{1\le i,j\le m}$ their intersections,
which is a set of disjoint sections of $Y$. For each $i=1,\dots,m$
we define a morphism $\pi_i:Y_i\to X_i$ as follows. If the action of $G$
on $Y_i$ is nontrivial we put $X_i:=Y_i/G$ and $\pi_i$ equal to the
quotient morphism. If the action of $G$ on $Y_i$ is trivial, note that
we must have $\mathrm{char}(K)=p$, since the map from $Y_i$ to its
image in $\cX$ is a $\cG$-torsor while $G_\cX\to \cG$ is an isomorphism
in characteristic $0$. Then we consider the Frobenius twist
$X_i:=Y_i^{(p)}$ and we define $\pi_i:Y_i\to X_i$ to be the relative
Frobenius. Finally we let $X$ be the scheme obtained by glueing the
$X_i$ along the sections $x_{i,j}=\pi_i(y_{i,j})\in X_i$ and
$x_{j,i}=\pi_j(y_{i,j})\in X_j$. There are markings $\Sigma_i^X\subset X$
given by the closures in $X$ of the generic markings $\Sigma_i^{X_\eta}$.
It is clear that the morphisms $\pi_i$ glue to a morphism $\pi:Y\to X$.

{\sc Extension of $\cX_\eta$ and $Y_\eta\to\cX_\eta$ along generic
nodes and markings.} In the following two lemmas we extend the stack
structure of $\cX_\eta$, and then the map $Y_\eta\to\cX_\eta$, along the
generic nodes and the markings:

\begin{lemma}
There is a unique extension $\ocX$ of the twisted curve $\cX_\eta$ over $X$,
such that $\ocX \to X$ is an isomorphism away from the generic nodes
and the markings.
\end{lemma}

\begin{proof}
We follow \cite[proof of proposition 4.3]{AOV2}.
First, let $\Sigma_{i,\eta}^{\cX_\eta}$ be a marking on $\cX_\eta$ and let
$P_{i,\eta}\subset Y_\eta$ be its preimage. There are
extensions $P_i\subset Y$ and $\Sigma_i^X\subset X$. Let $r$ be
the index of $\cX_\eta$ at $\Sigma_{i,\eta}^{\cX_\eta}$. Then we define
$\cX$ to be the stack of $r$-th roots of $\Sigma_i^X$ on $X$.
This extension is unique by the separatedness of stacks of $r$-th roots.

Now let $x_\eta\in X_\eta$ be a node with index $r$ and let $x\in X_s$
be its reduction. Locally in the {\'e}tale
topology, around $x$ the curve $X$ looks like the
spectrum of $R[u,v]/(uv)$. Let $B_u$ resp. $B_v$ be the branches at $x$
in $X$. The stacks of $r$-th roots of the divisor $u=0$ in $B_u$ an of
the divisor $v=0$ in $B_v$ are isomorphic and glue to give a stack
$\ocX$. By definition of $r$ we have $\ocX_\eta\simeq\cX_\eta$.
This extension is unique by the separatedness of stacks of $r$-th roots,
so the construction of $\ocX$ descends to $X$.
\end{proof}

\begin{lemma}
There is a unique lifting $Y \to \ocX$.
\end{lemma}

\begin{proof}
We need to check that there is a lifting at any point $y\in Y_s$ which
either lies on a marking or is the reduction of a generic node. We can
apply the purity lemma \cite[Lemma 4.4]{AOV2} provided that the local
fundamental group of $Y$ at $y$ is trivial and the local Picard group of
$Y$ at $y$ is torsion-free. In order to see this, we replace $R$ by its
strict henselization and $Y$ by the spectrum of the strict henselization
of the local ring at $y$. We let $U=Y\setminus \{y\}$.

If $y$ lies on a marking then $Y$ is isomorphic to the spectrum of
$R[a]^\sh$. Since this ring is local regular of dimension $2$, the
scheme $U$ has trivial fundamental group by the Zariski-Nagata purity
theorem, and trivial Picard group by Auslander-Buchsbaum. Hence the
purity lemma applies.

If $y$ is the reduction of a generic node, then $Y$ is isomorphic to
the strict henselization of $R[a,b]/(ab)$. Let $B_a=\Spec(R[a]^\sh)$
resp. $B_b=\Spec(R[b]^\sh)$ be the branches at $y$ and $U_a=U\cap B_a$,
$U_b=U\cap B_b$.

The schemes $U_a$ and $U_b$ have trivial fundamental group by
Zariski-Nagata, and they intersect in $Y$ in a single
point of the generic fibre. Moreover the map $U_a\,\sqcup\,U_b\to U$,
being finite surjective and finitely presented,
is of effective descent for finite {\'e}tale coverings
\cite[Exp. IX, cor. 4.12]{SGA1}. It then follows from the Van Kampen
theorem \cite[Exp. IX, th. 5.1]{SGA1} that $\pi_1(U)=1$.

For the computation of the local Picard group, first notice that
since $B_a,B_b$ are local regular of dimension $2$ we have
$\pic(U_a)=\pic(U_b)=0$, and moreover it is easy to see that
$H^0(U_a,\cO_{U_a}^\times)=R^\times$ and similarly for $U_b$.
Now we consider the long exact sequence in cohomology associated to the
short exact sequence
$$
0\to\cO_U^\times\to i_{a,*}\cO_{U_a}^\times\oplus i_{a,*}\cO_{U_a}^\times
\to i_{ab,*}\cO_{U_{ab}}^\times\to 0
$$
where the symbols $i_?$ stand for the obvious closed immersions.
We obtain
\begin{align*}
\pic(U) & 
=\mathrm{coker}
\big(H^0(U_a,\cO_{U_a}^\times)\oplus H^0(U_b,\cO_{U_a}^\times)\to
H^0(U_{ab},\cO_{U_{ab}}^\times)\big) \\
& =K^\times/R^\times = \ZZ,
\end{align*}
which is torsion-free as desired.
\end{proof}

Note that we still need to introduce stack structure over special
nodes of $\ocX$.

{\sc Extension of $\cG_\eta$ over generic points of $\ocX_s$.}
Let $\xi$ be the generic point of a component of $\ocX_s$. Let
$U$ be the localization of $\ocX$ at $\xi$ and $V$ be its inverse
image in $Y$. Consider the closure $\cG_\xi$ of $\cG_\eta$ in
$\Aut_UV$. 

\begin{proposition}\label{Raynaud-Romagny}
The scheme $\cG_\xi\to U$ is a finite flat group scheme of degree $p$,
and $V\to U$ is a $\cG_\xi$-torsor.
\end{proposition}

\begin{proof}
This is a generalization of \cite[Proposition 1.2.1]{Raynaud}, see
\cite[Theorem 4.3.5]{Romagny}. 
\end{proof}

{\sc Extension of $\cG_\eta$ over the smooth locus of $\ocX/S$.}
Quite generally, for a stable $p$-torsor $(\cX,\cG,Y)$ over a
scheme $T$, by $\Aut_\cX Y$ we denote the algebraic stack whose
objects over an $T$-scheme $U$ are pairs $(u,f)$ with $u\in\cX(U)$
and $f$ a $U$-automorphism of $Y\times_\cX U$.
Now consider $\ocX{}^\sm$, the smooth locus of $\ocX/S$, and its inverse image
$Y^\sm$ in $Y$. Then $Y^\sm \to \ocX{}^\sm$ is flat. Let $\cG^\sm$ be the
closure of $\cG_\eta$ in $\Aut_{\ocX{}^\sm} Y^\sm$. 

\begin{proposition} \label{extension_to_sm_locus}
The scheme $\cG^\sm\to \ocX{}^\sm$ is a finite flat group scheme of
degree $p$, and $Y^\sm\to \ocX{}^\sm$ is a $\cG^\sm$-torsor.
\end{proposition}

\begin{proof}
Given Proposition \ref{Raynaud-Romagny}, and since $\ocX{}^\sm$ has local
charts $U\to\ocX{}^\sm$ with $U$ regular $2$-dimensional, this follows
from \cite[Propositions 2.2.2 and 2.2.3]{A-Lubin}.
\end{proof}

{\sc Extension of $\cG^\sm$ over generic nodes of $\cX/S$.}
Consider the complement $\ocX{}^0$ of the isolated nodes of $\ocX_s$, and
its inverse image 
$Y^0$ in $Y$. 

\begin{lemma}
The morphism $Y^0 \to \ocX{}^0$ is flat.
\end{lemma}

\begin{proof}
It is enough to verify the claim at the reduction $x_s$ of an
arbitrary generic node $x_\eta\in X_\eta$. Since generic nodes remain
distinct in reduction, it is enough to prove that $Y\to \cX$ is
flat at a chosen point $y_s\in Y$ above $x_s$. Since the branches at
$y_s$ are not exchanged by $\cG$, {\'e}tale locally $Y$ and $\cX$ are the
union of two branches which are flat over $S$ and the restriction of
$Y\to\cX$ to each of the branches at $x_s$ is flat. Since proper
morphisms descend flatness (\cite{EGA}, IV.11.5.3) it follows
that $Y\to \cX$ is flat at $y_s$.
\end{proof}

Let $\cG^0$ be the closure of $\cG^\sm$ in $\Aut_{\ocX{}^0} Y^0$.

\begin{proposition}
  The stack $\cG^0\to \ocX{}^0$ is a finite flat group scheme of
  degree $p$, and  $Y^0\to \ocX{}^0$  is a $\cG^0$ torsor.
\end{proposition}

\begin{proof}
We only have to look around the closure of a generic node. Again since
proper morphisms descend flatness, it is enough to prove the claim
separately on the two branches. Then the result follows again from
\cite[Propositions 2.2.2 and 2.2.3]{A-Lubin} by the same reason as
in the proof of \ref{extension_to_sm_locus}.
\end{proof}

{\sc Twisted structure at special nodes.} Let $P$ be a special node of $X$.
By \cite[Section 3.2]{A-Lubin} there is a canonical twisted structure $\cX$
at $P$ determined by the local degree of $Y/X$. If near a given node
$Y_\eta/X_\eta$ is inseparable, then this degree is $p$. Otherwise  $Y/X$
has an action of $\ZZ/p\ZZ$ which is nontrivial near $P$, and therefore
the local degree is either 1 or $p$. Then $\cX$ is twisted with index $p$
at $P$ whenever this local degree is $p$. These twisted structures at the
various nodes $P$ glue to give a twisted curve $\cX$. 

We claim that
this $\cX$ is unique up to a unique isomorphism. This follows from
Proposition \ref{Prop:twisting-proper} below. Indeed, let $a$ be the
node function which to a node $P$ of $X$ gives the local degree of
$Y/X$ at $Y$, and let $r_i$ be the fixed indices at the sections. Then
the stack of $a$-twisted curves over $X/S$ with markings of indices
$r_i$ is proper over $S$, hence $\cX$ is uniquely determined by $\cX_\eta$ up
to unique isomorphism.  

By
\cite[Lemma 3.2.1]{A-Lubin}, there is a unique lifting $Y \to \cX$, and by
\cite[Theorem 3.2.2]{A-Lubin} the group scheme $\cG^0$ extends
uniquely to $\cG \to \cX$ such that $Y$ is a $\cG$-torsor. The
rigidification extends immediately by taking the closure, since $\cG
\to \cX$ is finite.
\end{proof}

\section{Examples}\label{Sec:examples}

\subsection{First, some non-examples}\label{Sec:nonex} Consider a smooth projective
curve $X$ of genus $h>1$ in  
characteristic $p$ and and a $p$-torsion point in its Jacobian,
corresponding to a 
$\mu_p$-torsor $Y' \to X$. This is {\em not} a stable $p$-torsor in the sense of
Definition~\ref{Def:stable-torsor}: the curve $Y'$ is necessarily singular.
In fact, $Y'\to X$ may be described by a locally logarithmic differential
form $\omega$ on $X$, such that if locally $\omega=df/f$ for some
$f\in\cO_X^\times$ 
then $Y'$ is given by an equation $z^p=f$. Since the genus $h>1$, all
differentials on $X$ have zeroes, and each zero of $\omega$ (i.e. a
zero of the derivative 
of $f$ with respect to a coordinate) contributes to a unibranch
singularity on $Y'$.

Now consider a ramified $\ZZ/p\ZZ$-cover $Y\to X$ of smooth projective curves
over a field. 
Let $y\in Y$ be a fixed point for the action of $\ZZ/p\ZZ$ and let
$x\in X$ be its 
image. In characteristic $0$, since the stabilizer of $y$ is a
multiplicative group, 
the curve $X$ may be twisted at $x$ to yield a stable $\ZZ/p\ZZ$-torsor
$Y\to\cX$. However in 
characteristic $p$ the stabilizer is additive and the result is not a
$\ZZ/p\ZZ$-torsor. Hence ramified
covers of smooth curves in characteristic $p$ do not provide stable
$\ZZ/p\ZZ$-torsors. 

However something else does occur in both
examples: the torsor $Y'\to X$ of the first example, and the branched
cover  $Y \to X$ in
the second, lift to characteristic 0. The reduction back to characteristic $p$ of the
corresponding stable torsor ``contains the original cover'' in the
following sense: there is a unique component $\cX$ whose coarse moduli space
is isomorphic to $X$. In particular that component $\cX$ is
necessarily a twisted curve, and the group scheme over it has to
degenerate to $\alpha_p$ over the twisted points.  We see a
manifestation of this in the next example.  
\Dan{Added this discussion}

\subsection{Limit of a $p$-isogeny of elliptic curves} Now consider the case where $X$ is an
elliptic curve, with a marked point $x$, over a discrete valuation ring $R$ of characteristic~0
and residue characteristic $p$. For simplicity assume that $R$ contains $\mu_p$; let
$\eta$ be the generic point of $\Spec R$ and  $s$  the closed point of $\Spec R$. Given a
$p$-torsion point on $X$ with non-trivial reduction, we obtain a corresponding nontrivial
$\mu_p$-isogeny $Y'\to X$. Over the generic point $\eta$ we can make $Y'_\eta$
stable by marking the fiber $P_\eta$ over $x_\eta$. But note that the reduction of $P_\eta$
in $Y$ is not \'etale, hence something must modified. Since our stack is proper, a
stable $p$-torsor $Y \to \cX$ limiting $Y'_\eta \to X_\eta$ exists, at least over a base change
of $R$. Here is how to describe it.

 Consider the completed local ring
$\cO_{Y',O}\simeq R[[Z]]$ at the origin $O\in Y'_s$ and its spectrum $\DD$. Then $\DD_\eta$
is identified with an open $p$-adic disk modulo Galois action. Write
$P_\eta=\{P_{\eta,1},\dots,P_{\eta,p}\}$ as a sum of points permuted by the $\mu_p$-action.
Then the $P_{\eta,i}$ induce $K$-rational points of $\DD_\eta$ which moreover are
$\pi$-adically
equidistant, i.e. the valuation $v=v_\pi(P_{\eta,i}-P_{\eta,j})$ is independent of $i,j$.
It follows that after blowing-up the closed subscheme with ideal $(\pi^v,Z)$ these points
reduce to $p$ distinct points in the exceptional divisor.
Thus after twisting at the node, the fiber $Y_s \to \cX_s$ over the special point $s$
of $R$ is described  as follows: \Dan{added diagram}

$$\xymatrix{
Y_s \ar[d]\ar@{=}[r] & Y_s' \cup \PP^1 \ar[d] & P \ar[d]\ar@{_(->}[l]\\
X_s'\ar@{=}[r] & E \cup Q & \{0\}\ar@{_(->}[l]
}$$ 

Here 
\begin{itemize}
\item $Y_s$ is a union of two components $Y'_s\cup \PP^1$, attached at the origin of $Y'_s$.
\item $\cX_s$ is a twisted curve with two components $E \cup Q$
\item Here $E = X_s(\sqrt[p]{x})$ and $Q = \PP^1(\sqrt[p]{\infty})$, with the twisted points attached.
\item   The map $Y_s\to\cX_s$ decomposes into $Y'_s \to E$ and $\PP^1 \to Q$.
\item    $\PP^1 \to Q$ is an Artin--Schreier cover ramified at $\infty$.
\item The curve is marked by the inverse image of $0\in Q$  in $\PP^1$, which is a
$\ZZ/p\ZZ$-torsor $P\subset\PP^1$.
\item The map    $Y'_s \to E$ is a lift of $Y'_s \to X_s$.
\item The group scheme $\cG \to \cX$ is generically \'etale on $Q$ and generically
$\mu_p$ on $E$, but the fiber over the node is $\alpha_p$.  
\end{itemize}

Notice that we can view $Y'_s \to E$, marked by the origin on $Y'_s$, as a twisted torsor
as well, but this twisted torsor does not lift to characteristic 0 simply because
the marked point on $Y'_s$ can not be lifted to an invariant
divisor. This is an example of the phenomenon described at the end of
Section \ref{Sec:nonex} above.
\Dan{Added sentence}

A very similar picture occurs when the cover $Y'_\eta \to X_\eta$ degenerates to an
$\alpha_p$-torsor. If, however, the reduction of the cover is a $\ZZ/p\ZZ$-torsor, then
$Y'\to X$, marked by the fiber over the origin, is already stable and new components do
not appear.

\subsection{The double cover of $\PP^1$ branched over 4 points}
Consider an elliptic  double cover $Y$ over $\PP^1$ in characteristic 0 given by the
equation $y^2 = x(x-1)(x-\lambda)$. Marked by the four branched points, it becomes a
stable $\mu_2$-torsor over the twisted curve $Q= \PP^1(\sqrt{0,1,\infty,\lambda})$.
What is its reduction in characteristic 2? We describe here one case,
the others can be described in a similar way.
 
If the elliptic curve $Y$ has good ordinary reduction $E_s$, the
picture is as follows: 
 $Y_s$ has three components
$\PP^1\cup E_s\cup \PP^1$. 
 The twisted curve $\cX_s$ also has three rational
components $Q_1\cup Q_2\cup Q_3$. The map splits as $\PP^1\to Q_1$, $E_s \to Q_2$ and
$\PP^1\to Q_3$, where the first and last are generically $\mu_2$-covers, and $E_s\to Q_2$
is a lift of the hyperelliptic cover $E_s \to \PP^1$. The fibers of $\cG$ at the nodes of
$X_s$ are both $\alpha_2$. The points $0,1,\infty,\lambda$ reduce to two pairs, one pair
on each of the two $\PP^1$ components, for instance:
$$\xymatrix{
 & \PP^1\cup E_s\cup \PP^1\ar[d]&\\
\{0,1\} \ar@{^(->}[r] & Q_1\cup Q_2\cup Q_3 & \{\lambda,\infty\}.\ar@{_(->}[l]
}$$ 


\appendix

\section{Group schemes of order $p$} \label{complements_order_p}

In this appendix, we give some complements on group schemes of order $p$.
The main topic is the construction of an embedding of a given group scheme
of order $p$ into an affine smooth one-dimensional group scheme (an analogue
of Kummer or Artin-Schreier theory). Although not strictly necessary in the paper,
this result highlights the nature of our stable torsors in two respects~: firstly
because the original definition of generators in \cite[\S~1.4]{KM} involves a
smooth ambient group scheme, and secondly because the short exact sequence given
by this embedding induces a long exact sequence in cohomology that may be useful
for computations of torsors.

Anyway, let us now state the result.

\begin{definition}
Let $G\to S$ be a finite locally free group scheme of order~$p$.
\begin{enumerate}
\item A {\em generator} is a morphism of $S$-group schemes
$\gamma:(\ZZ/p\ZZ)_S\to G$ such that the sections $x_i=\gamma(i)$,
$0\le i\le p-1$, are a full set of sections.
\item A {\em cogenerator} is a morphism of $S$-group schemes
$\kappa:G\to\mu_{p,S}$ such that the Cartier dual $(\ZZ/p\ZZ)_S\to G^\vee$
is a generator.
\end{enumerate}
\end{definition}

We will prove the following.

\begin{theorem} \label{thm_embedding}
Let $S$ be a scheme and let $G\to S$ be a finite locally free group
scheme of order $p$. Let $\kappa:G\to\mu_{p,S}$ be a cogenerator.
Then $\kappa$ can be canonically inserted into a commutative diagram
with exact rows
$$
\xymatrix{
0 \ar[r] & G \ar[r] \ar[d]^\kappa & \sG \ar[r]^{\varphi_\kappa} \ar[d]
& \sG' \ar[d] \ar[r] & 0 \\
0 \ar[r] & \mu_{p,S} \ar[r] & \GG_{m,S} \ar[r]^p & \GG_{m,S}
\ar[r] & 0}
$$
where $\varphi_\kappa:\sG\to \sG'$ is an isogeny between
affine smooth one-dimensional $S$-group schemes with geometrically
connected fibres.
\end{theorem}

In order to obtain this, we introduce two categories of invertible sheaves
with sections: one related to groups with a cogenerator and
one related to groups defined as kernels of isogenies, and we compare these
categories.

\begin{remark}
Not all group schemes of order $p$ can be embedded into an affine
smooth group scheme as in the theorem. For example, assume that there
exists a closed immersion from $G=(\ZZ/p\ZZ)_\QQ$ to some
affine smooth one-dimensional geometrically connected $\QQ$-group
scheme $\sG$. Then $\sG$ is a form of $\GG_{m,\QQ}$ and $G$ is its
$p$-torsion subgroup. Since $\sG$ is trivialized by a quadratic
field extension $K/\QQ$, we obtain $G_K\simeq\mu_{p,K}$. This
implies that
$K$ contains the $p$-th roots of unity, which is impossible for $p>3$.
Similar examples can be given for $\ZZ/p\ZZ$ over the Tate-Oort ring
$\Lambda\otimes\QQ$.
\end{remark}

\subsection{Tate-Oort group schemes}

We recall the notations and results of the Tate-Oort classification of
group schemes of degree $p$ over the ring $\Lambda$ (section~2 of
\cite{Tate-Oort}). We introduce two fibered categories:
\begin{itemize}
\item a $\Lambda$-category $TG$ of {\em triples} encoding
{\em groups},
\item a $\Lambda$-category $TGC$ of {\em triples} encoding
{\em groups with a cogenerator}.
\end{itemize}
Let $\chi:\FF_p\to\ZZ_p$ be the unique multiplicative section of the
reduction map, that is $\chi(0)=0$ and if $m\in\FF_p^\times$ then
$\chi(m)$ is the $(p-1)$-st root of unity with residue equal to $m$.
Set
$$
\Lambda=\ZZ[\chi(\FF_p),\frac{1}{p(p-1)}]\cap \ZZ_p.
$$
There is in $\Lambda$ a particular element $w_p$ equal to $p$ times a
unit.

\begin{definition}
The category $TG$ is the category fibered over $\Spec \Lambda$
whose fiber categories over a $\Lambda$-scheme $S$ are as follows.
\begin{itemize}
\item Objects are the triples $(L,a,b)$ where $L$ is an invertible
sheaf and $a\in\Gamma(S,L^{\otimes (p-1)})$,
$b\in\Gamma(S,L^{\otimes (1-p)})$ satisfy
$a\otimes b=w_p1_{\cO_S}$.
\item Morphisms between $(L,a,b)$ and $(L',a',b')$ are the morphisms
of invertible sheaves $f:L\to L'$, viewed as global sections of
$L^{\otimes -1}\otimes L'$, such that $a\otimes f^{\otimes p}=f\otimes a'$
and $b'\otimes f^{\otimes p}=f\otimes b$.
\end{itemize}
\end{definition}

The main result of \cite{Tate-Oort} is an explicit description of a
covariant equivalence of fibered categories between $TG$ and the category of
finite locally free group schemes of order $p$. The group scheme
associated to a triple $(L,a,b)$ is denoted $G^L_{a,b}$.
Its Cartier dual is isomorphic to $G^{L^{-1}}_{b,a}$.

\begin{examples}
We have $(\ZZ/p\ZZ)_S=G^{\cO_S}_{1,w_p}$ and $\mu_{p,S}=G^{\cO_S}_{w_p,1}$.
Moreover if $G=G^L_{a,b}$ then a morphism $(\ZZ/p\ZZ)_S\to G$ is given by
a global section $u\in\Gamma(S,L)$ such that $u^{\otimes p}=u\otimes a$
and a morphism $G\to \mu_{p,S}$ is given by a global section
$v\in\Gamma(S,L^{-1})$ such that $v^{\otimes p}=v\otimes b$.
\end{examples}

\begin{lemma} \label{gen_cogen}
Let $S$ be a $\Lambda$-scheme and let $G=G^L_{a,b}$ be a finite locally
free group scheme of rank $p$ over $S$. Then:
\begin{enumerate}
\item Let $\gamma:(\ZZ/p\ZZ)_S\to G$ be a morphism of $S$-group schemes
given by a section $u\in\Gamma(S,L)$ such that $u^{\otimes p}=u\otimes a$.
Then $\gamma$ is a generator if and only if $u^{\otimes (p-1)}=a$.
\item Let $\kappa:G\to\mu_{p,S}$ be a morphism of $S$-group schemes
given by a section $v\in\Gamma(S,L^{-1})$ such that
$v^{\otimes p}=v\otimes b$. Then $\kappa$ is a cogenerator if and only
if $v^{\otimes (p-1)}=b$.
\end{enumerate}
\end{lemma}

\begin{proof}
The proof of (2) follows from (1) by Cartier duality so we only deal with
(1). The claim is local on $S$ so we may assume that $S$ is affine equal to
$\Spec(R)$ and $L$ is trivial. It follows from \cite{Tate-Oort} that
$G=\Spec R[x]/(x^p-ax)$ and the section
$\gamma(i):\Spec(R)\stackrel{i}{\to}(\ZZ/p\ZZ)_R\to G$ is given by the
morphism of algebras $R[x]/(x^p-ax)\to R$, $x\mapsto \chi(i) u$. Thus
$\gamma$ is a generator
if and only if $\mathrm{Norm}(f)=\prod f(\chi(i)u)$ for all functions
$f=f(x)$. In particular for $f=1+x$ one finds $\mathrm{Norm}(f)=(-1)^pa+1$
and $\prod (1+\chi(i)u)=(-1)^pu^{p-1}+1$. Therefore if $\gamma$ is a generator
then $u^{p-1}=a$. Conversely, assuming that $u^{p-1}=a$ we want to prove that
$\mathrm{Norm}(f)=\prod f(\chi(i)u)$ for all $f$. It is enough to prove this
in the universal case where $R=\Lambda[a,b,u]/(ab-w_p,u^p-u)$. Since $a$ is
not a zerodivisor in $R$, it is in turn enough to prove the equality after
base change to $K=R[1/a]$. Then $G_K$ is \'etale and the morphism
$$K[x]/(x^p-ax)=K[x]/\prod (x-\chi(i)u)\to K^p$$
taking $f$ to the
tuple $(f(\chi(i)u))_{0\le i\le p-1}$ is an isomorphism of algebras. Since
the norm in $K^p$ is the product of the coordinates, the result follows.
\end{proof}

\begin{definition}
The category $TGC$ is the category fibered over $\Spec \Lambda$
whose fibers over a $\Lambda$-scheme $S$ are as follows.
\begin{itemize}
\item Objects are the triples $(L,a,v)$
where $L$ is an invertible sheaf and $a\in\Gamma(S,L^{\otimes (p-1)})$,
$v\in\Gamma(S,L^{\otimes -1})$ satisfy $a\otimes v^{\otimes (p-1)}=w_p1_{\cO_S}$.
\item Morphisms between $(L,a,v)$ and $(L',a',v')$ are the morphisms
of invertible sheaves $f:L\to L'$, viewed as global sections of
$L^{\otimes -1}\otimes L'$, such that $a\otimes f^{\otimes p}=f\otimes a'$
and $v'\otimes f=v$.
\end{itemize}
\end{definition}

By lemma~\ref{gen_cogen}, the category $TGC$ is equivalent to the
category of group schemes with a cogenerator. The functor from group
schemes with a cogenerator to group schemes that forgets the
cogenerator is described in terms of categories of invertible sheaves
by the functor $\omega:TGC \to TG$ given by
$\omega(L,a,v)=(L,a,v^{\otimes (p-1)})$.

Note also that lemma~\ref{gen_cogen} tells us that for any locally free
group scheme $G$ over a $\Lambda$-scheme $S$, there exists a finite
locally free morphism $S'\to S$ of degree $p-1$ such that $G\times_S S'$
admits a generator or a cogenerator.

\subsection{Congruence group schemes}

Here, we introduce and describe a $\ZZ$-category $TCG$ of {\em triples}
encoding {\em congruence groups}.

Let $R$ be ring with a discrete valuation $v$ and let $\lambda\in R$ be
such that $(p-1)v(\lambda)\le v(p)$. In \cite{SOS} are introduced some
group schemes $H_\lambda=\Spec R[x]/(((1+\lambda x)^p-1)/\lambda^p)$
with multiplication $x_1\star x_2=x_1+x_2+\lambda x_1x_2$. (The notation
in {\em loc. cit.} is $\sN$.) Later Raynaud called them {\em congruence
groups of level $\lambda$} and we will follow his terminology.
We now define the analogues of these group schemes over a general base.
The objects that are the input of the construction constitute the
following category.

\begin{definition}
The category $TCG$ is the category fibered over $\Spec \ZZ$
whose fibers over a scheme $S$ are as follows.
\begin{itemize}
\item Objects are the triples $(M,\lambda,\mu)$ where $M$ is an invertible
sheaf over $S$ and the global sections $\lambda\in\Gamma(S,M^{-1})$ and
$\mu\in\Gamma(S,M^{p-1})$ are subject to the condition
$\lambda^{\otimes (p-1)}\otimes \mu=p1_{\cO_S}$.
\item Morphisms between $(M,\lambda,\mu)$ and $(M',\lambda',\mu')$ are
morphisms of invertible sheaves $f:M\to M'$ viewed as sections of
$M^{-1}\otimes M'$ such that $f\otimes \lambda'=\lambda$ and
$f^{\otimes (p-1)}\otimes \mu=\mu'$.
\end{itemize}
\end{definition}

We will exhibit a functor $(M,\lambda,\mu)\leadsto H^M_{\lambda,\mu}$
from $TCG$ to the category of group schemes, with $H^M_{\lambda,\mu}$
defined as the kernel of a suitable isogeny.

First, starting from $(M,\lambda)$ we construct a smooth affine
one-dimensional group scheme denoted $\sG^{(M,\lambda)}$, or simply
$\sG^{(\lambda)}$. We see $\lambda$ as a morphism $\lambda:\VV(M)\to \GG_{a,S}$
of (geometric) line bundles over $S$, where $\VV(M)=\Spec \mathrm{Sym}(M^{-1})$
is the (geometric) line bundle associated to $M$. We define $\sG^{(\lambda)}$
as a scheme by the fibered product
$$
\xymatrix{
\sG^{(\lambda)} \ar[r]^{1+\lambda} \ar[d] & \GG_{m,S} \ar[d] \\
\VV(M) \ar[r]^{1+\lambda} & \GG_{a,S} \ . }
$$
The points of $\sG^{(\lambda)}$ with values in an $S$-scheme $T$
are the global sections $u\in\Gamma(T,M\otimes\cO_T)$
such that $1+\lambda\otimes u$ is invertible. We endow $\sG^{(\lambda)}$
with a multiplication given on the $T$-points by
$$
u_1\star u_2=u_1+u_2+\lambda\otimes u_1\otimes u_2 \ .
$$
The zero section of $\VV(M)$ sits in $\sG^{(\lambda)}$ and is
the unit section for the law just defined. The formula
$$
(1+\lambda\otimes u_1)(1+\lambda\otimes u_2)=1+\lambda\otimes (u_1\star u_2)
$$
shows that $1+\lambda:\sG^{(\lambda)}\to\GG_{m,S}$ is a morphism of group
schemes. Moreover, if the locus where $\lambda:\VV(M)\to \GG_{a,S}$
is an isomorphism is scheme-theoretically dense, then
$\star$ is the unique group law on $\sG^{(\lambda)}$ for which this holds.
This construction is functorial in $(M,\lambda)$: given a morphism of
invertible sheaves $f:M\to M'$, in other words a global section of
$M^{-1}\otimes M'$, such that $f\otimes \lambda'=\lambda$, there is
a morphism $f:\sG^{(\lambda)}\to\sG^{(\lambda')}$ making the diagram
$$
\xymatrix{
\sG^{(\lambda)} \ar[r]^{1+\lambda} \ar[d]_f & \GG_{m,S} \\
\sG^{(\lambda')} \ar[ru]_{1+\lambda'} & \\}
$$
commutative. The notation is coherent since that morphism is indeed
induced by the extension of $f$ to the sheaves of symmetric algebras.

Then, we use the section $\mu\in\Gamma(S,M^{p-1})$ and the relation
$\lambda^{\otimes (p-1)}\otimes \mu=p1_{\cO_S}$ to define an isogeny
$\varphi$ fitting into a commutative diagram
$$
\xymatrix{
\sG^{(\lambda)} \ar[r]^>>>>>{\varphi} \ar[d]_{1+\lambda} &
\sG^{(\lambda^{\otimes p})} \ar[d]^{1+\lambda^{\otimes p}} \\
\GG_{m,S} \ar[r]^{{}^{\wedge}p} & \GG_{m,S} \ . }
$$
The formula for $\varphi$ is given on the $T$-points
$u\in\Gamma(T,M\otimes\cO_T)$ by
$$
\varphi(u)=u^{\otimes p}+\sum_{i=1}^{p-1} \bin{p}{i}\,
\lambda^{\otimes (i-1)}\otimes\mu\otimes u^{\otimes i}
$$
where $\bin{p}{i}=\frac{1}{p}{p \choose i}$ is the binomial coefficient
divided by $p$. In order to check that the diagram is commutative and
that $\varphi$ is an isogeny, we may work locally on $S$ hence we
may assume that $S$ is affine and that $M=\cO_S$. In this
case, the two claims follow from the universal case i.e. from points
(1) and (2) in the following lemma.

\begin{lemma} \label{lemma_universal_case}
Let $\cO=\ZZ[E,F]/(E^{p-1}F-p)$ and let $\lambda,\mu\in\cO$ be the images
of the indeterminates $E,F$. Then, the polynomial
$$P(X)=X^p+\sum_{i=1}^{p-1} \bin{p}{i}\,\lambda^{i-1}\mu X^i \:\in\cO[X]$$
satisfies:
\begin{enumerate}
\item $1+\lambda^pP(X)=(1+\lambda X)^p$, and
\item $P(X+Y+\lambda XY)=P(X)+P(Y)+\lambda^pP(X)P(Y)$.
\end{enumerate}
\end{lemma}

\begin{proof}
Point (1) follows by expanding $(1+\lambda X)^p$ and using the fact that
$p=\lambda^{p-1}\mu$ in $\cO$. Then we compute:
\begin{align*}
1+\lambda^pP(X+Y+\lambda XY) & =(1+\lambda(X+Y+\lambda XY))^p \\
& =(1+\lambda X)^p(1+\lambda Y)^p \\
& =(1+\lambda^pP(X))(1+\lambda^pP(Y)) \\
&=1+\lambda^p(P(X)+P(Y)+\lambda^pP(X)P(Y)).
\end{align*}
Since $\lambda$ is a nonzerodivisor in $\cO$, point (2) follows.
\end{proof}

\begin{definition}
We denote by $H^M_{\lambda,\mu}$ the kernel of $\varphi$, and call it the
{\em congruence group scheme} associated to $(M,\lambda,\mu)$.
\end{definition}

This construction is functorial in $(M,\lambda,\mu)$. Precisely, consider
two triples $(M,\lambda,\mu)$ and $(M',\lambda',\mu')$ and a morphism
of invertible sheaves $f:M\to M'$ viewed as a section of $M^{-1}\otimes M'$
such that $f\otimes \lambda'=\lambda$ and $f^{\otimes (p-1)}\otimes \mu=\mu'$.
Then we have morphisms $f:\sG^{(\lambda)}\to\sG^{(\lambda')}$
and $f^{\otimes p}:\sG^{(\lambda^{\otimes p})}\to\sG^{(\lambda'^{\otimes p})}$
compatible with the isogenies $\varphi$ and $\varphi'$, and $f$
induces a morphism $H^M_{\lambda,\mu}\to H^{M'}_{\lambda',\mu'}$.
Note also that the image of $H^M_{\lambda,\mu}$ under
$1+\lambda:\sG^{(\lambda)}\to\GG_{m,S}$ factors through $\mu_{p,S}$,
so that by construction $H^M_{\lambda,\mu}$ comes embedded into a diagram
$$
\xymatrix{
0 \ar[r] & H^M_{\lambda,\mu} \ar[r] \ar[d]^\kappa & \sG^{(\lambda)}
\ar[r] \ar[d]^{1+\lambda}
& \sG^{(\lambda^{\otimes p})} \ar[d]^{1+\lambda^{\otimes p}} \ar[r] & 0 \\
0 \ar[r] & \mu_{p,S} \ar[r] & \GG_{m,S} \ar[r] & \GG_{m,S}
\ar[r] & 0 . }
$$
The formation of this diagram is also functorial.

\begin{lemma}
The morphism $\kappa:H^M_{\lambda,\mu}\to\mu_{p,S}$ is a cogenerator.
\end{lemma}

\begin{proof}
We have to show that the dual map $(\ZZ/p\ZZ)_S\to (H^M_{\lambda,\mu})^\vee$
is a generator. This means verifying locally on $S$ certain equalities
of norms. Hence we may assume that $S$ is affine and that $M$ is trivial,
then reduce to the universal case where $S$ is the spectrum of the
ring $\cO$ with elements $\lambda,\mu$ satisfying $\lambda^{p-1}\mu=p$ as in
lemma~\ref{lemma_universal_case}, and finally restrict to the schematically
dense open subscheme $S'=D(\lambda)\subset S$. Since
$\sG^{(\lambda)}\times_S S'\to\GG_{m,S'}$ is an isomorphism, then
$H^M_{\lambda,\mu}\times_S S'\to\mu_{p,S'}$ and the dual morphism also are.
The claim follows immediately.
\end{proof}

\subsection{Equivalence between $TGC$ and $TCG\otimes_\ZZ \Lambda$}

The results of the previous subsection imply that for a $\Lambda$-scheme $S$,
a triple $(M,\lambda,\mu)\in TCG(S)$ gives rise in a functorial way to a
finite locally free group scheme with cogenerator
$\kappa:H^M_{\lambda,\mu}\to\mu_{p,S}$, that is, an object of $TGC(S)$.

\begin{theorem} \label{TCG_equiv_TGC}
The functor
$$
F:TCG\otimes_\ZZ \Lambda\to TGC
$$
defined above is an equivalence of fibered categories over $\Lambda$.
If $(M,\lambda,\mu)$ has image 
$(L,a,v)$ then $H^M_{\lambda,\mu}\simeq G^L_{a,v^{\otimes (p-1)}}$.
\end{theorem}

\begin{proof}
The main point is to describe $F$ in detail using the Tate-Oort
classification, and to see that it is essentially surjective. The
description of the action of $F$ on morphisms and the verification
that it is fully faithful offers no difficulty and will be omitted.

Let $(M,\lambda,\mu)$ be a triple in $TCG(S)$ and let $G=H^M_{\lambda,\mu}$.
We use the notations of section 2 of \cite{Tate-Oort}, in particular
the structure of the group $\mu_p$ is described by a function $z$,
the sheaf of $\chi$-eigensections $J=y\cO_S\subset\cO_{\mu_p}$ with
distinguished generator $y=(p-1)e_1(1-z)$, and constants
$$w_1=1,w_2,\dots,w_{p-1},w_p=pw_{p-1}\in\Lambda.$$
The augmentation ideal of the algebra $\cO_G$ is the sheaf $I$ generated
by $M^{-1}$, and by \cite{Tate-Oort} the
subsheaf of $\chi$-eigensections is the sheaf $I_1=e_1(I)$
where $e_1$ is the $\cO_S$-linear map defined in \cite{Tate-Oort}.
It is an invertible sheaf and $L$ is (by definition) its inverse.

We claim that in fact $I_1=e_1(M^{-1})$. In order to see this, we may
work locally. Let $x$ be a local generator for $M^{-1}$ and let
$$
t:=(p-1)e_1(-x) \in I_1.
$$
Let us write $\lambda=\lambda_0 x$ for some local function $\lambda_0$.
We first prove that
$$
(\star) \quad\quad
x=\frac{1}{1-p}\big(t+\frac{\lambda_0 t^2}{w_2}+\dots
+\frac{\lambda_0^{p-2}t^{p-1}}{w_{p-1}}\big).
$$
In fact, by construction the map $\cO_{\mu_p}\to\cO_G$ is given
by $z=1+\lambda_0 x$, so we get $y=(p-1)e_1(1-z)=\lambda_0 t$.
In order to check the expression for $x$ in terms of $t$, we can
reduce to the universal case (lemma~\ref{lemma_universal_case}).
Then $\lambda_0$ is not a zerodivisor and we can harmlessly multiply
both sides by $\lambda_0$. In this form, the equality to be
proven is nothing else than the identity (16) in \cite{Tate-Oort}.
Now write $t=\alpha t^*$ with $t^*$ a local generator for $I_1$ and
$\alpha$ a local function. Using~($\star$) we find that
$x=\alpha x^*$ for some $x^*\in\cO_G$. Since $x$
generates $M^{-1}$ in the fibres over $S$, this proves that $\alpha$
is invertible. Finally $t$ is a local generator for $I_1$ and this
finishes the proof that $I_1=e_1(M^{-1})$.

Let $x^\vee$ be the local generator for $M$ dual to $x$ and write
$\mu=\mu_0(x^\vee)^{\otimes (p-1)}$ for some local function $\mu_0$
such that $(\lambda_0)^{p-1}\mu_0=p$. Let $t^\vee$ be the local
generator for $L$ dual to $t$. We define a local section $a$ of
$L^{\otimes (p-1)}$ by
$$
a=w_{p-1}\mu_0(t^\vee)^{\otimes (p-1)}
$$
and a local section $v$ of $L^{-1}$ by
$$
v=\lambda_0 t.
$$
These sections are independent of the choice of the local generator
$x$, because if $x'=\alpha x$ then
$$
(x')^\vee=\alpha^{-1}x^\vee
\ ; \
t'=\alpha t
\ ; \
(t')^\vee=\alpha^{-1}t^\vee
\ ; \
\lambda'_0=\alpha^{-1}\lambda_0
\ ; \
\mu'_0=\alpha^{p-1}\mu_0
$$
so that
$$
a'=w_{p-1}\mu'_0(t'^\vee)^{\otimes (p-1)}
=w_{p-1}\alpha^{p-1}\mu_0\alpha^{1-p}(t^\vee)^{\otimes (p-1)}=a
$$
and
$$
v'=\lambda'_0 t'=\alpha^{-1}\lambda_0\alpha t=v.
$$
They glue to global sections $a$ and $v$ satisfying
$$
a\otimes v^{\otimes (p-1)}=w_p1_{\cO_S}.
$$
Let us prove that $a$ and $v$ are indeed the sections defining
$G$ and the cogenerator in the Tate-Oort classification. The
verification for $a$ amounts to checking that the relation
$$
t^p=w_{p-1}\mu_0 t
$$
holds in the algebra $\cO_G$. This may be seen in the universal
case where $\lambda_0$ is not a zerodivisor, hence after multiplying
by $(\lambda_0)^p$ this follows from the equality $y^p=w_py$ from
\cite{Tate-Oort}. The verification for $v$ amounts to noting that
the cogenerator $G\to\mu_{p,S}$ is indeed given by $y\mapsto v$.

This completes the description of $F$ on objects.
Finally we prove that $F$ is essentially surjective. Assume given
$(L,a,v)$ and let $t$ be a local generator for $I_1=L^{-1}$.
Write $a=w_{p-1}\mu_0(t^\vee)^{\otimes (p-1)}$, $v=\lambda_0 t$
and define an element $x\in\cO_G$ by the expresion ($\star$) above.
If we change the generator $t$ to another $t'=\alpha t$, then
$\lambda'_0=\alpha^{-1}\lambda_0$ and $x'=\alpha x$. It follows that
the subsheaf of $\cO_G$ generated by $x$ does not depend on the
choice of the generator for $I_1$, call it $N$. Reducing to the
universal case as before, we prove that $t=(p-1)e_1(-x)$. This
shows that in fact $N$ is an invertible sheaf and we take $M$ to
be its inverse. Finally we define sections $\lambda\in\Gamma(S,M^{-1})$
and $\mu\in\Gamma(S,M^{\otimes (p-1)})$ by the local expressions
$\lambda=\lambda_0 x$ and $\mu=\mu_0(x^\vee)^{\otimes (p-1)}$. It
is verified like in the case of $a,v$ before that they do not
depend on the choice of $t$ and hence are well-defined global
sections. The equality
$\lambda^{\otimes (p-1)}\otimes \mu=p1_{\cO_S}$ holds true and the
proof is now complete.
\end{proof}

\subsection{Proof of theorem~\ref{thm_embedding}}

We are now in a position to prove theorem~\ref{thm_embedding}.
We keep its notations. Since the construction of the isogeny
$\varphi_\kappa$ and the whole commutative diagram is canonical,
if we perform it after fppf base change $S'\to S$ then it will
descend to $S$. We choose $S'=S_1\amalg S_2$ where
$S_1=S\otimes_\ZZ\ZZ[1/p]$ and $S_2=S\otimes_\ZZ\Lambda$.
Over $S_1$ the group scheme $G$ is \'etale and the cogenerator is
an isomorphism by \cite[lemma 1.8.3]{KM}. We take
$\sG=\sG'=\GG_{m,S}$ and $\varphi_\kappa$ is the $p$-th power map.
Over $S_2$ we use theorem~\ref{TCG_equiv_TGC} which provides a
canonical isomorphism between $\kappa$ and $H^M_{\lambda,\mu}$
with its canonical cogenerator, embedded into a diagram of the
desired form. This completes the proof.

%
%
%
%
%
%

\section{Weil restriction of closed subschemes} \label{Sec:Weil}

Let $Z\to X$ be a morphism of $S$-schemes (or algebraic spaces) and
denote by $h:X\to S$ the structure map. The Weil restriction $h_*Z$
of $Z$ along $h$ is the functor on $S$-schemes defined by
$(h_*Z)(T)=Hom_X(X\times_S T,Z)$. It may be seen as a left adjoint
to the pullback along $h$, or as the functor of sections of $Z\to X$.

If $Z\to X$ is a closed immersion of schemes (or algebraic spaces)
of finite presentation over $S$, there are two main cases
where $h_*Z$ is known to be representable by
a closed subscheme of $S$. As is well-known, this has applications
to representability of various equalizers, kernels, centralizers,
normalizers, etc. These two cases are~:

\medskip

(i) if $X\to S$ is proper flat and $Z\to S$ is separated, by
the Grothendieck-Artin theory of the Hilbert scheme,

\medskip

(ii) if $X\to S$ is essentially free, by
\cite[Exp.~VIII, th.~6.4]{sga3}.

\medskip

In this appendix, we want to prove that $h_*Z$ is representable by
a closed subscheme of $S$ in a case that includes both situations
and is often easier to check in practice, namely the case where
$X\to S$ is flat and pure.

\subsection{Essentially free and pure morphisms}
\label{EF_and_pure}

We recall the notions of essentially free and pure morphisms
and check that essentially free morphisms and proper morphisms
are pure.

In \cite[Exp.~VIII, section~6]{sga3},
a morphism $X\to S$ is called
{\em essentially free} if and only if there exists a covering
of $S$ by open affine subschemes $S_i$, and for each $i$ an
affine faithfullay flat morphism $S'_i\to S_i$ and a covering
of $X'_i=X\times_S S'_i$ by open affine subschemes $X'_{i,j}$
such that the function ring of $X'_{i,j}$ is free as a module
over the function ring of $S'_i$.

In fact, the proof of theorem 6.4 in \cite[Exp.~VIII]{sga3}
works just as well with a slightly weaker notion than freeness
of modules. Namely, for a module
$M$ over a ring $A$, let us say that $M$ is {\em good} if the
canonical map $M\to M^{\vee\vee}$ from $M$ to its linear bidual
is injective after any change of base ring $A\to A'$. It is a
simple exercise to see that this is equivalent to $M$ being a
submodule of a product module $A^I$ for some set $I$, over $A$
and after any base change $A\to A'$. For instance, free modules,
projective modules, product modules are good. This gives rise
to a notion of {\em essentially good} morphism, and in particular
{\em essentially projective} morphism. Then inspection
of the proof of theorem 6.4 of \cite[Exp.~VIII]{sga3} shows
that it remains valid for these morphisms.

In \cite[3.3.3]{R-G}, a morphism locally of finite type $X\to S$
is called {\em pure} if and only if for all points $s\in S$,
with henselization $(\tilde S,\tilde s)$, and all points
$\tilde x\in \tilde X$ where $\tilde X=X\times_S \tilde S$,
if $\tilde x$ is an associated point in its fibre then its
closure in $\tilde X$ meets the special fibre. Examples of
pure morphisms include proper morphisms (by the valuative
criterion for properness) and morphisms locally
of finite type and flat, with geometrically irreducible fibres
without embedded components (\cite[3.3.4]{R-G}).

Finally if $X\to S$ is locally of finite presentation and essentially
free, then it is pure. Indeed, with the notations above
for an essentially free morphism, one sees using \cite[3.3.7]{R-G}
that it is enough to see that for each $i,j$ the scheme $X'_{i,j}$
is pure over $S'_i$. But since the function ring of $X'_{i,j}$
is free over the function ring of $S'_i$, this follows from
\cite[3.3.5]{R-G}.

\subsection{Representability of $h_*Z$}

\begin{proposition}
Let $h:X\to S$ be a morphism of finite presentation,
flat and pure, and let $Z\to X$ be a closed immersion. Then
the Weil restriction $h_*Z$ is representable by a closed
subscheme of $S$.
\end{proposition}

\begin{proof}
The question is local for the {\'e}tale topology on $S$. Let $s\in S$
be a point and let $\cO^h$ be the henselization of the local
ring at $s$. By \cite[3.3.13]{R-G}, for each $x\in X$ lying over
$s$, there exists an open affine subscheme $U_x^h$ of
$X\times_S\Spec(\cO^h)$ containing $x$ and whose function
ring is free as an $\cO^h$-module. Since $X_s$ is quasi-compact,
there is a finite
number of points $x_1,\dots,x_n$ such that the open affines
$U_i^h=U_{x_i}^h$ cover it. Since $X$ is locally of
finite presentation, after restricting to an {\'e}tale neighbourhood
$S'\to S$ of $s$, there exist affine open subschemes $U_i$ of $X$
inducing the $U_i^h$. According to \cite[3.3.8]{R-G}, the
locus of the base scheme $S$ where $U_i\to S$ is pure is
open, so after shrinking $S$ we may assume that for each $i$ the
affine $U_i$ is flat and pure. This means that its function ring
is projective by \cite[3.3.5]{R-G}. In other words, the union
$U=U_1\cup\dots\cup U_n$ is essentially projective over $S$ in
the terms of the comments in \ref{EF_and_pure}. If $k:U\to X$
denotes the structure map, it follows from theorem~6.4 of
\cite[Exp. VIII]{sga3} that $k_*(Z\cap U)$ is representable
by a closed subscheme of $S$. On the other hand, according to
\cite[3.1.7]{Romagny}, replacing $S$ again by a smaller
neighbourhood of $s$, the open immersion $U\to X$ is
$S$-universally schematically dense. One deduces immediately
that the natural morphism $h_*Z\to k_*(Z\cap U)$ is an isomorphism.
This finishes the proof.
\end{proof}

This proposition has a long list of corollaries and applications
listed in \cite[Exp.~VIII, section~6]{sga3}. In particular let
us mention the following~:

\begin{corollary} \label{equalizer}
Let $X\to S$ be a morphism of finite presentation, flat and pure
and $Y\to S$ a separated morphism. Consider two morphisms
$f,g:X\to Y$. Then the condition $f=g$ is represented by a closed
subscheme of $S$.
\end{corollary}

\begin{proof}
Apply the previous proposition to the pullback of the diagonal
of $Y$ along $(f,g):X\to Y\times_S Y$.
\end{proof}




\end{document}